\documentclass[12pt]{huber_article}

\usepackage{amsthm,amsmath}
\RequirePackage[colorlinks,citecolor=blue,urlcolor=blue]{hyperref}

\usepackage{booktabs}
\usepackage{dsfont}
\usepackage{amsfonts}
\usepackage{tikz}
\theoremstyle{plain}
\newtheorem{theorem}{Theorem}[section]
\newtheorem{lemma}{Lemma}[section]
\theoremstyle{definition}
\newtheorem{example}{Example}[section]
\newtheorem{definition}{Definition}[section]

\newcommand{\prob}{\mathbb{P}}
\newcommand{\unifdist}{\textsf{Unif}}
\newcommand{\berndist}{\textsf{Bern}}
\newcommand{\expdist}{\textsf{Exp}}
\newcommand{\ind}{{\mathds{1}}}
\newcommand{\gk}{{\arabic{line})}\stepcounter{line}}
\newcounter{line}

\begin{document}


\title{The Fundamental Theorem of Perfect \\ 
  Simulation\footnote{Supported by
  U.S. National Science Foundation grant DMS-1418495}}

\author{Mark Huber \\ {\tt mhuber@cmc.edu}}

\maketitle

\begin{abstract}
Here several perfect simulation algorithms are brought under a single 
framework, and shown to derive from the same probabilistic result, called
here the Fundamental Theorem of Perfect Simulation (FTPS).
An exact simulation algorithm has output according to an input distribution
$\pi$.  Perfect simulations are a subclass of exact simulations where
recursion is used either explicitly or implicitly.  The FTPS gives
two simple criteria that, when satisfied, give a correct perfect simulation
algorithm.
First the algorithm must terminate in finite time
with probability 1.  Second, the algorithm must be locally correct in the 
sense that the algorithm can be proved correct given the assumption that
any recursive call used returns an output from the correct distribution.
This simple idea is surprisingly powerful.  Like other general techniques
such as Metropolis-Hastings for approximate simulation, the FTPS allows
for the flexible construction of existing and new perfect simulation 
protocols.
This theorem can be used to verify the correctness of many perfect simulation
protocols, including Acceptance Rejection,
Coupling From the Past, and Recursive Bernoulli factories.
\end{abstract}


\section{Introduction}

Perfect simulation algorithms generate random variates exactly
from a target distribution using a random number of steps. They are
typically used for the same types as problems where Markov chain Monte
Carlo algorithms are employed, such as Bayesian posterior inference and 
approximation algorithms for \#P-complete problems.  Perfect simulation
algorithms are still the only methods known 
to give practical algorithms for exact simulation
from high-dimensional examples such as the Ising model~\cite{proppw1996}.

There exist many different protocols for building perfect simulation
algorithms, including several variations of acceptance/rejection (AR) and
coupling from the past (CFTP)~\cite{proppw1996}.  There are also more
specialized algorithms such as the Recursive Bernoulli 
factory~\cite{huber2016a}.

The purpose of this work is to bring all of these methods under a common
mathematical framework. 
Each of these methods
can be individually proved to be correct. The proofs 
(as well as the algorithmic
stucture) of these protocols share common features, and the goal of this
work is to identify the most important common feature. Once isolated, this
notion is intuitively very compelling, and in fact it is not difficult 
to show that
this feature gives correctness of the algorithm.  With this notion in place,
it becomes straightforward to show correctness of probabilisitic recursive
algorithms such as the Bernoulli factory.

As with Markov chain Monte Carlo approaches, the idea is simple, but the 
applications wide-ranging.  What separates perfect simulation algorithms
from other algorithms is the use of recursion.  That is, after taking a 
step in the process, the algorithm typically calls itself again, perhaps
with different parameter inputs.  

In other words, at each step, the algorithm makes random choices and
transforms the problem into one of simulating from a new distribution
that depends on the random choices made.  If the new distribution puts
probability 1 on a single state, then call this a {\em halting} distribution,
as the algorithm need merely output that single state and then terminate.

To have a correct perfect simulation algorithm for a target distribution
$\pi$, it must be necessary that the algorithm
halts with probability 1.  It turns out that this necessary condition is 
also sufficient:  any recursive probabilistic algorithm that halts with
probability 1 will output from $\pi$ if the recursive distributions are
chosen appropriately.  This is the Fundamental Theorem of perfect 
simulation (FTPS).

\begin{theorem}[Fundamental Theorem of perfect simulation]
\label{THM:FTPS}
Consider the following algorithm
  \begin{center}
\setcounter{line}{1}
\begin{tabular}{rl}
\toprule
\multicolumn{2}{l}{{\sc Perfect\_Simulation}} \\
\multicolumn{2}{l}{\textit{Input: } $\pi_i$, $i$} \\
\midrule
\gk & \hspace*{0em} Draw $U_i \leftarrow \nu_i$ \\
\gk & \hspace*{0em} If $U_i \in A_i$ \\
\gk & \hspace*{1em}   Return $f_i(U_i)$ \\
\gk & \hspace*{0em} Else \\
\gk & \hspace*{1em}   Draw $X \leftarrow 
                            \textsc{Perfect\_Simulation}(\pi_{U_i},i+1)$ \\
\gk & \hspace*{0em}   Return $g_i(X,U_i)$ \\
\bottomrule
\end{tabular}
\end{center}
For all $i$ let 
\begin{equation}
\label{EQN:FTPS}
X_i = f_i(U_i) \ind(U_i \in A_i) + g_i(Y_i,U_i) \ind(U_i \notin A_i),
\end{equation}
where $Y_i \sim \pi_{U_i}$.  If $X_i \sim \pi_i$ for all $i$ and 
the algorithm halts with probability 1, then the output of 
$\textsc{Perfect\_Simulation}(\pi_0,0)$ is exactly $\pi_0$.
\end{theorem}

Another way to view the FTPS is that when designing a recursive simulation
algorithm, it is only necessary to verify that the algorithm is locally
correct and halts with probability 1.  In other words, you can assume
that the recursive call correctly generates a draw from the desired distribution
in proving that the algorithm works.  In any other context this would be
circular reasoning, but for 
recursive simulation algorithms that run in finite time
(with probability 1), this is enough to guarantee global 
correctness of the algorithm.

This FTPS idea was first introduced in a text (\cite{huber2015b}).  Here we
generalize the notion as first introduced and expand its application
to several problems that do not appear in~\cite{huber2015b}.
The remainder of the paper is organized as follows.  
Section~\ref{SEC:examples} gives examples of perfect simulation protocols
that fall into this framework, and uses the FTPS to show correctness.
Section~\ref{SEC:proof} then proves the FTPS and discusses various
interpretations.

\section{Perfect simulation protocols}
\label{SEC:examples}

This section shows how FTPS implies the correctness of several
perfect simulation methodologies.

\subsection{Acceptance/rejection}

The acceptance/rejection (AR) protocol (also known as rejection sampling)
was the first widely used perfect simulation method.
\begin{example}[AR for a five sided fair die]
   Suppose that it is possible to draw independently identically 
   distributed samples from a fair six-sided die, and the goal is to 
   simulate uniformly from $\{1,2,3,4,5\}$.
  \begin{center}
\setcounter{line}{1}
\begin{tabular}{rl}
\toprule
\multicolumn{2}{l}{{\sc AR\_example}} \\
\midrule
\gk & \hspace*{0em} Draw $X \leftarrow \unifdist(\{1,2,3,4,5,6\})$ \\
\gk & \hspace*{0em} If $X \in \{1,2,3,4,5\}$ \\
\gk & \hspace*{1em}   return $X$ and halt \\
\gk & \hspace*{0em} Else \\
\gk & \hspace*{1em}   $X \leftarrow \textsc{AR\_example}$ \\
\gk & \hspace*{1em}   return $X$ and halt \\
\bottomrule
\end{tabular}
\end{center}
Here $\unifdist(A)$ denotes the uniform distribution over the set $A$.

This is an example of a recursive algorithm:  it might call itself in
the course of execution.  Note that the algorithm will halt with probability 1.
Because the recursive call to the algorithm is the same as the original
call, it is 
easy to prove correctness.  For $i \in \{1,2,3,4,5\}$ and output $X$:
\begin{align*}
 \prob(X = i) &= 1/6 + (1/6)\prob(X = i).
\end{align*}
Solving then gives $\prob(X = i) = 1/5$ as desired.

The FTPS can be applied to this example as follows.
The target distribution is  
$\pi \sim \unifdist(\{1,\ldots,5\})$.
Using the notation
of the pseudocode in Theorem~\ref{THM:FTPS}, for all $i$, 
$\nu_i \sim \unifdist(\{1,\ldots,6\})$, $A_i = \{1,\ldots,5\}$, 
$\pi_{U_i} \sim \pi$, $f_i(u) = u$ and $g_i(x,u) = x$.  Since at each step
the probability of halting is $5/6$, with probability 1 the algorithm
terminates.  Let
\[
X_i = f_i(U_i)\ind(U_i \in A) + g_i(Y_i,U_i) \ind(U_i \notin A)
 = U_i \ind(U_i \leq 5) + Y_i \ind(U_i = 6)
\]
where $Y_i \sim \pi$.  For $i \in \{1,\ldots,5\}$,
\[
\prob(X_i = i) = \prob(U_i = i) + \prob(Y_i = i)\prob(U_i = 6) = 
 (1/6)+(1/5)(1/6) = 1/5,
\]
so the algorithm is locally correct.
Global correctness
follows immediately from the FTPS.
\end{example}

Any algorithm expressible in pseudocode can also be represented graphically as
a branching process, and Figure~\ref{FIG:fig1} shows this representation
for \textsc{AR\_Example}.
In the figure, $\delta(U)$ represents the Diract delta function that puts all
the probability mass on $U$.  That is,
for $[X|U] \sim \delta(U)$, $\prob(X = U) = 1$.

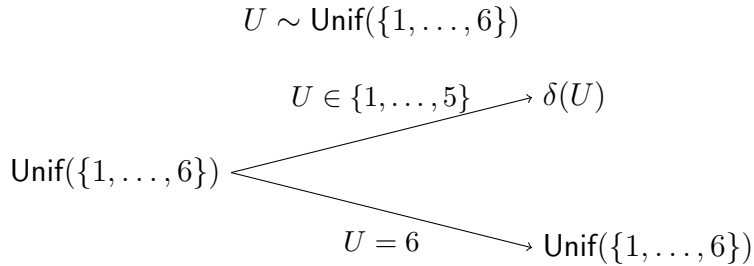
\begin{figure}[ht]
  \begin{center}
    \begin{tikzpicture}
      \draw (2,2) node {$U \sim \unifdist(\{1,\ldots,6\})$};
      \draw[->] (0,0) node[left] {$\unifdist(\{1,\ldots,6\})$} -- 
        node[above=4pt] {\small $U \in \{1,\ldots,5\}$} 
        (4,1) node[right] {$\delta(U)$};
      \draw[->] (0,0) -- node[below=4pt] {\small $U = 6$} (4,-1) node[right] 
           {$\unifdist(\{1,\ldots,6\})$};
    \end{tikzpicture}
  \end{center}
  \caption{Branching process representation of \textsc{AR\_Example}.}
  \label{FIG:fig1}
\end{figure}

\begin{example}[General AR]
  Suppose that $\nu$ is a measure over $B$, and $A$ is a measurable
  subset of $B$ such that $\nu(A) > 0$.  Then general AR samples
  $X \sim \nu$ conditioned to lie in $A$.
  \begin{center}
\setcounter{line}{1}
\begin{tabular}{rl}
\toprule
\multicolumn{2}{l}{{\sc General\_AR}} \\
\midrule
\gk & \hspace*{0em} Draw $X \leftarrow \nu$ \\
\gk & \hspace*{0em} If $X \in A$  \\
\gk & \hspace*{1em}   return $X$ and halt \\
\gk & \hspace*{0em} Else \\
\gk & \hspace*{1em}   $X \leftarrow \textsc{General\_AR}$ \\
\gk & \hspace*{1em}   return $X$ and halt \\
\bottomrule
\end{tabular}
\end{center}
Then as with the simple example, for all $i$, set $\nu_i \sim \nu$,
$A_i = A$, $f_i(u) = u$, and $g_i(x,u) = x$.  Then since 
the chance of halting at each step is $\nu(A) > 0$, the algorithm halts
with probability 1, and FTPS gives correctness.
\end{example}

Careful use of AR
can result in polynomial time algorithms even for high-dimensional examples.
For instance, 
AR can be used to 
sample from weighted permutations for approximating
the permanent of dense matrices~\cite{huber2006a,huber2008a} or 
satisfying assignments of disjuntive normal forms.  For both these 
applications the associated counting problem is 
\#P-complete~\cite{valiant1979,valiant1979b,jerrums1989}

\subsection{Coupling from the past}

There do exist high-dimensional problems where 
the running time of basic AR grows exponentially
with the distribution, thereby rendering the protocol impractical for
these models.  

A canonical example of this is the Ising model.  This model takes
a graph $(V,E)$, and assigns weight $w(x)$ to $x \in \{-1,1\}^V$ of 
$\exp(-\beta H(x))$, where $\beta$ is a parameter of the model, and 
$H(x) = -\sum_{\{i,j\} \in E} x(i)x(j)$ is called the Hamiltonian.  The
probability distribution then becomes,
\[
\pi(\{x\}) = \frac{w(x)}{Z_\beta}, \quad
  Z_\beta = \sum_{y \in \{-1,1\}^V} w(y).
\]

For many decades, dealing with distributions like the Ising model through the 
use of Markov chains 
to generate approximately correct samples was the only method available.
A Markov chain with a particular stationary distribution 
is implemented in a computer simulation via a stationary
update function.
\begin{definition}
  Call $\phi:\Omega \times [0,1]$ a {\em stationary update function} for 
  distribution $\pi$ over $\Omega$ if for $X \sim \pi$ and 
  $U \sim \unifdist([0,1])$, $\phi(X,U) \sim \pi$ as well.
\end{definition}

If $U_0,U_1,\ldots$ are iid $\unifdist([0,1])$, then 
setting $X_0=x_0$ and $X_{t+1} =\phi(X_t,U_t)$ creates a Markov chain,
and it is well known that under mild conditions the distribution of 
$X_t$ will approach $\pi$.  Unfortunately, it is very difficult to 
determine bounds on how quickly the chain approaches distribution $\pi$,
called the {\em mixing time} of the Markov chain.

For some chains, there might exist positive probability that $\phi(x,U)$
is the same state for all $x \in \Omega$.  When $\#\phi(\Omega,U) = 1$, say
that the state space has {\em completely coupled} or {\em coalesced}.

\begin{example}[Coupling from the past]
CFTP~\cite{proppw1996} is 
a perfect simulation protocol designed to use stationary update 
functions to generate samples exactly from the distribution $\pi$.
A recursive formulation of the algorithm is as follows.  Here,
suppose $A$ is a set such that for $u \in A$, $\phi(\Omega,u)$
is a set that consists of only a single state.
  \begin{center}
\setcounter{line}{1}
\begin{tabular}{rl}
\toprule
\multicolumn{2}{l}{{\sc CFTP}} \\
\midrule
\gk & \hspace*{0em} Draw $U \leftarrow \unifdist([0,1])$ \\
\gk & \hspace*{0em} If $U \in A$ \\
\gk & \hspace*{1em}   return the unique element of $\phi(\Omega,U)$ and halt \\
\gk & \hspace*{0em} Else \\
\gk & \hspace*{1em}   $X \leftarrow \textsc{CFTP}$ \\
\gk & \hspace*{1em}   return $\phi(X,U)$ and halt \\
\bottomrule
\end{tabular}
\end{center}
This falls nicely in the FTPS framework.  Here for all $i$,
$\nu_i \sim \unifdist([0,1])$,
$A_i = A$, $f_i(u)$ is the unique element of $\phi(\Omega,u)$, and 
$g_i(x,u) = \phi(x,u)$.
To use the FTPS, it must be true that $\prob(\#\phi(\Omega,U) = 1) > 0$.
Then CFTP terminates with probability 1 and FTPS gives correctness.
\end{example}

Often, a single step in a Markov chain is not enough to have positive 
probability of coalescence.  Note
that for a fixed $t$, composing $\phi$ with itself $t$ times also gives
a stationary update.  Let $\phi_t = \phi \circ \phi \circ \cdots \circ \phi$
denote this $t$-fold 
composition.  The hope is if $t$ is large enough, then 
$\phi_t$ might coalesce.

Like with the mixing time of the chain, finding $t$ exactly is not 
usually possible.  Therefore, Propp and Wilson~\cite{proppw1996} suggested
doubling $t$ at each recursive step.

\begin{example}[Doubling CFTP]
  This method is described using the following pseudocode.
  \begin{center}
\setcounter{line}{1}
\begin{tabular}{rl}
\toprule
\multicolumn{2}{l}{{\sc Doubling\_CFTP}} \\
\multicolumn{2}{l}{\textit{Input: } $t$} \\
\midrule
\gk & \hspace*{0em} Draw $U \leftarrow \unifdist([0,1]^t)$ \\
\gk & \hspace*{0em} If $U \in A$ \\
\gk & \hspace*{1em}   return the unique element of $\phi_t(\Omega,U)$ 
  and halt \\
\gk & \hspace*{0em} Else \\
\gk & \hspace*{1em}   $X \leftarrow \textsc{Doubling\_CFTP}(2t)$ \\
\gk & \hspace*{1em}   return $\phi(X,U)$ and halt \\
\bottomrule
\end{tabular}
\end{center}
Here $\nu_i \sim \unifdist([0,1]^{2^i})$, $f_i(u)$ is the unique 
element of $\phi(\Omega,u)$, $g_i(x,u) = \phi_{2^i}(\Omega,U)$.  
To prove the algorithm terminates with probability 1, it suffices
for there to exist $t$ such that coalescence occurs using $\phi_t$ 
with positive probability.  

With this condition in place, the FTPS says that to prove global correctness,
one can assume that the recursive call in line 6 returns output from the 
correct distribution.  
Assuming $X$ from line 6 is
drawn from $\pi$, then for a measurable set $B$, the probability
that the output of the algorithm is in $B$ is 
\begin{align*}
  \prob(\phi_t(\Omega,U) &\in B|U \in A)\prob(U \in A)
 + \prob(\phi_t(X,U) \in B|U \notin A)\prob(U \notin A) \\
 &= \prob(\phi_t(X,U) \in B) = \pi(B)
\end{align*}
\end{example}

In~\cite{huber2008b}, a more sophisticated scheme for altering $t$
was used to guarantee that the probability that the running time 
was much larger than the mean time decreased exponentially.  Although
the scheme was more complex, it also fits the framework of FTPS and so
correctness immediately follows.

\subsection{Recursive Bernoulli Factory}

Bernoulli factories were introduced in~\cite{asmussengt1992} as
a subroutine needed for perfect simulation from the stationary distribution
of regenerative processes.  Work on constructing efficient and
practical Bernoulli
factories has continued since~\cite{keaneo1994,nacup2005,
latuszynskikpr2011,hubertoappearc}.

A Bernoulli factory works as follows.  Suppose that an iid sequence of Bernoulli
random variables $B_1,B_2,\ldots \sim \berndist(p)$ are available but
$p$ itself is unknown.
The goal is to build a new random variable $X \sim \berndist(p)$ as a 
function of the $\{B_i\}$ together with external randomness 
$U \sim \unifdist([0,1])$ 
that uses as few coin flips as possible.

\begin{definition}
Given $p^* \in (0,1]$ and a function $f:[0,p^*] \rightarrow [0,1]$, 
a {\em Bernoulli factory} is a computable 
function $\cal A$ that takes as input a number $u \in [0,1]$ together
with a sequence of values in $\{0,1\}$, 
and returns an output in $\{0,1\}$
where the following holds.  For any $p \in [0,p^*]$, 
$X_1,X_2,\ldots$ iid $\berndist(p)$, and $U \sim \unifdist([0,1])$,
let $T$ be the infimum of times $t$ such that the value of 
${\cal A}(U,X_1,X_2,\ldots)$ only depends on the values of $X_1,\ldots,X_t$.
Then
\begin{enumerate}
\item{$T$ is a stopping time with respect to the natural
filtration and $\prob(T < \infty) = 1$.}
\item{${\cal A}(U,X_1,X_2,\ldots) \sim \berndist(f(p))$.}
\end{enumerate}
Call $T$ the {\em running time} of the Bernoulli factory.
\end{definition}

Using the perfect simulation notation from earlier, a Bernoulli
factory algorithm is a perfect simulation algorithm for 
$\berndist(f(p))$ such that for all $i$, 
$\nu_i \in \berndist(p)^k \times\unifdist([0,1])$ for some nonnegative
integer $k$.
The state space for $\berndist(f(p))$, is 
$\{0,1\}$, and so it 
holds that $f_i(u) \in \{0,1\}$ and $g_i(x,u) \in \{0,1\}$
for all $i$.

In other words, all of the distributions employed by the perfect simulation
algorithm also must be Bernoulli distributions.  That means that to 
check~\eqref{EQN:FTPS}, it suffices to show that the probability 
the left hand side equals 1 equals the probability the right hand side
is 1, greatly simplifying the calculations.

\begin{example}[Von Neumann's Bernoulli factory]
Von Neumann~\cite{vonneumann1951} constructed a simple Bernoulli factory
where $f(p) = 1/2$ for all $p$.  It utilized two flips of the coin at
each level of recursion, and is represented graphically in 
Figure~\ref{FIG:fig2}.
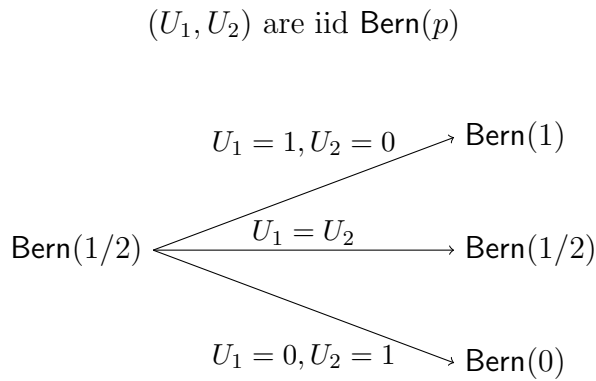
\begin{figure}[ht]
  \begin{center}
    \begin{tikzpicture}
      \begin{scope}[scale=1]
      \draw (2,3) node {$(U_1,U_2)$ are iid $\berndist(p)$};
      \draw[->] (0,0) node[left] {$\berndist(1/2)$} --
        node[above=10pt] {\small $U_1 = 1, U_2 = 0$} 
        (4,1.5) node[right] {$\berndist(1)$};
      \draw[->] (0,0) -- node[below=10pt] {\small $U_1 = 0, U_2 = 1$} 
      (4,-1.5) node[right] 
           {$\berndist(0)$};
      \draw[->] (0,0) -- node[above=-2pt] {\small $U_1 = U_2$} (4,0) node[right]
        {$\berndist(1/2)$};
      \end{scope}
    \end{tikzpicture}
  \end{center}
  \caption{Branching process representation of the Von Neumann constant 
    Bernouli factory.}
  \label{FIG:fig2}
\end{figure}
At each level of the recursion there is a $2p(1-p)$ chance of halting,
so for $p\in(0,1)$ the algorithm terminates in finite time with probability 1.
Moreover,
\[
1/2 = p(1-p)(1) + [p^2 + (1 - p)^2](1/2) + (1-p)p(0),
\]
so the local correctness condition is satisfied.  The algorithm is therefore
correct by the FTPS.
\end{example}

\begin{example}[Exponential Bernoulli factory]
In~\cite{beskospr2006} showed how to build a Bernoulli factory for 
$f(p) = \exp(-p)$, which was needed as part of a method for perfectly
simulating from diffusions.  They created such a factory using a 
thinned Poisson process.

Consider here the slightly more general
problem of drawing from $f(p) = \exp(-Cp)$, where $C$ is a known 
positive constant.  Then using a single coin flip together with
an exponential random variable, the algorithm for this factory is
represented in Figure~\ref{FIG:fig3}.

\begin{figure}[ht]
  \begin{center}
    \begin{tikzpicture}
      \begin{scope}[scale=1]
      \draw (2,3) node {$U_1 \sim \berndist(p)$, $U_2 \sim \expdist(C)$};
      \draw[->] (0,0) node[left] {$\berndist(\exp(-Cp)$} --
        node[above=10pt] {\small $U_2 \geq 1$} 
        (4,1.5) node[right] {$\berndist(1)$};
      \draw[->] (0,0) -- node[below=10pt] 
           {\small $U_2 < 1, U_1 = 0$} 
      (4,-1.5) node[right] 
           {$\berndist(\exp(-C(1-U_2)p)$};
      \draw[->] (0,0) -- 
           (4,0) node[right]
        {$\berndist(0)$};
      \draw (2.5,0.2) node {\small $U_2 < 1, U_1 = 1$} ;
      \end{scope}
    \end{tikzpicture}
  \end{center}
  \caption{Branching process representation of an exponential factory.}
  \label{FIG:fig3}
\end{figure}
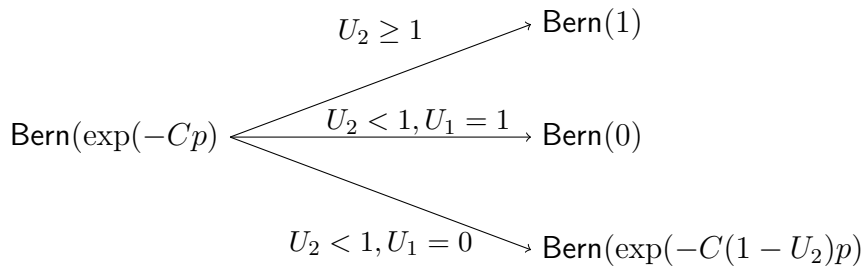

The probability that $U_2 \geq 1$ is $\exp(-C)$.  If $U_2 < 1$ and 
$U_1 = 0$, then the probability of a 1 becomes $\exp(-C(1-U_2)p)$.  
Also, $\prob(U_1 = 0) = 1 - p$.  Therefore,
the right hand side of equation~\eqref{EQN:FTPS} is:
\begin{align*}
\exp(-C) + (1 - p)\int_{u_2 \in [0,1]} C \exp(-C u_2)\exp(-C(1-u_2)p) \ du_2
\end{align*}
which evaluates to $\exp(-Cp)$ as desired.

At each recursive step, there is at least a $\exp(-C)$ chance of 
terminating and so the overall algorithm terminates with probability 1.
Therefore the FTPS immediately gives correctness.

\end{example}

\begin{example}[Linear Bernoulli factory]

The original application of Asmussen et. al~\cite{asmussengt1992} 
required Bernoulli factories of the form $f(p) = Cp$ for a constant
$p$.

Nacu and Peres~\cite{nacup2005} called a randomized algorithm with
random running time $T$ a 
{\em fast simulation} if there existed constants $M > 0$ and 
$\rho < 1$ such that $\prob(T > t) \leq M \rho^t$ for all $t > 0$.  
One of their results was that if $2p$ has a fast simulation, then
any function $f(p)$ that is real analytic over $(0,1)$ has a 
fast simulation.  The converse also holds:  any $f$ with a fast
simulation is real analytic on any open subset of its domain.

For these reasons, the $Cp$ Bernoulli factory is especially important.
The first provably polynomial expected time Bernoulli factory for $Cp$ coins was
introduced in~\cite{huber2016a}, and was an explicitly recursive 
perfect simulation algorithm.  It was shown there that the expected
number of coin flips needed was bounded above by
\[
9.5 C \epsilon^{-1}.
\]

It was also shown in~\cite{huber2016a} that any Bernoulli factory that
worked for all $p$ and $C$ such that $Cp \in [0,1-\epsilon]$ required 
at least
\[
0.004C\epsilon^{-1}
\]
flips of the coin on average.  Hence the algorithm of~\cite{huber2016a}
is the best possible up to the constant.

The Bernoulli factory of~\cite{huber2016a} actually solves the more general
problem of flipping a $(Cp)^i$ coin for any integer $i$.  Of course, 
$i = 1$ is the case of actual interest, but the factory works for any
integer $i \geq 1$.

This algorithm can be represented using
three types of recursions.  In the first recursion (Figure~\ref{FIG:fig4}),
a single $p$-coin is flipped.  If it is heads, then the algorithm
halts and outputs a 1, otherwise it changes the problem to flipping
a $(C-1)p/(1-p)$ coin.

\begin{figure}[ht]
  \begin{center}
    \begin{tikzpicture}
      \begin{scope}[scale=1]
      \draw (3,3) node {$U_1 \sim \berndist(p)$};
      \draw[->] (0,0) node[left] {$\berndist((Cp)^i)$} --
        node[above=10pt] {\small $U_1 = 1$} 
        (4,1.5) node[right] {$\berndist((Cp)^{i-1})$};
      \draw[->] (0,0) -- node[below=10pt] 
           {\small $U_1 = 0$} 
      (4,-1.5) node[right] 
           {$\berndist((Cp)^{i-1}(C-1)p/(1-p))$};
      \end{scope}
    \end{tikzpicture}
  \end{center}
  \caption{The first piece of the recursive Bernoulli factory.}
  \label{FIG:fig4}
\end{figure}

When $i = 0$, the goal is just to flip a $\berndist(1)$-coin, which
is always 1, and so this is a halting state.

The second piece attempts to turn a $(C-1)p/(1-p)$-coin flip
problem back into a $Cp$-coin flip problem.  This is done by flipping
a $(C-1)/C$-coin.  If heads, then it is necessary to flip one 
$Cp$-coin.  Otherwise, it is necessary to flip both one $Cp$-coin and
still one $(C-1)p/(1-p)$-coin.  This step can be repeated (a geometrically
distributed number of times) until the $(C-1)p/(1-p)$-coin flip is gone.
This is represented in Figure~\ref{FIG:fig5}.

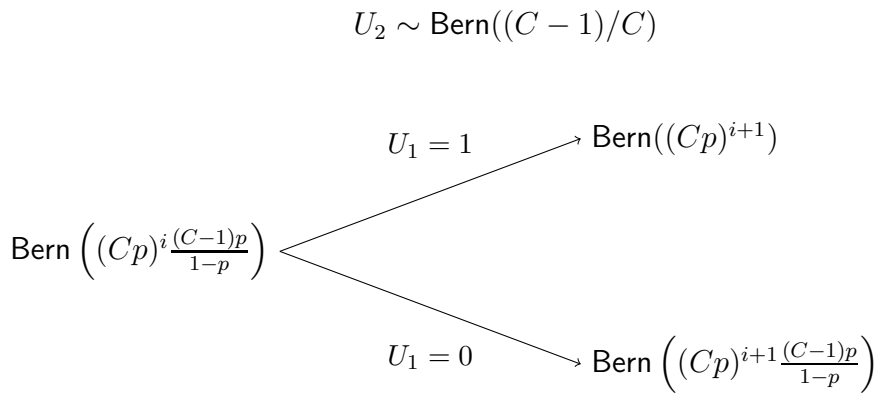
\begin{figure}[ht]
  \begin{center}
    \begin{tikzpicture}
      \begin{scope}[scale=1]
      \draw (3,3) node {$U_2 \sim \berndist((C-1)/C)$};
      \draw[->] (0,0) node[left] {$\berndist\left((Cp)^i 
        \frac{(C-1)p}{1-p}\right)$} --
        node[above=10pt] {\small $U_1 = 1$} 
        (4,1.5) node[right] {$\berndist((Cp)^{i+1})$};
      \draw[->] (0,0) -- node[below=10pt] 
           {\small $U_1 = 0$} 
      (4,-1.5) node[right] 
           {$\berndist\left((Cp)^{i+1}\frac{(C-1)p}{1-p}\right)$};
      \end{scope}
    \end{tikzpicture}
  \end{center}
  \caption{The second piece of the recursive Bernoulli factory.}
  \label{FIG:fig5}
\end{figure}

The third and final piece works for any function $g(p)$ and 
parameter $\alpha$ such that 
$g(p) \leq \alpha$.  It flips an $\alpha$-coin.  If this is 
tails, then the overall output is tails.  If it is heads, then
a $\alpha^{-1} g(p)$-coin must be flipped.

\begin{figure}[ht]
  \begin{center}
    \begin{tikzpicture}
      \begin{scope}[scale=1]
      \draw (3,3) node {$U_3 \sim \berndist(\alpha)$};
      \draw[->] (0,0) node[left] {$\berndist(g(p))$} --
        node[above=10pt] {\small $U_1 = 1$} 
        (4,1.5) node[right] {$\berndist(\alpha^{-1}g(p))$};
      \draw[->] (0,0) -- node[below=10pt] 
           {\small $U_1 = 0$} 
      (4,-1.5) node[right] 
           {$\berndist(0)$};
      \end{scope}
    \end{tikzpicture}
  \end{center}
  \caption{The last piece of the recursive Bernoulli factory.}
  \label{FIG:fig6}
\end{figure}

These pieces are combined as follows.  Begin with $i = 1$, and use
the first piece to either move to $i=0$ (which halts) or to a 
$(C-1)p/(1-p)$.  Use the second piece to replace the $(C-1)p/(1-p)$-coin
with a geometric (with mean $C/(C-1)$) number of $Cp$-coins.  Continue
until $i = 0$ or $i \geq 4.6\epsilon^{-1}$.  At this point, for
$Cp \leq 1 - \epsilon$, $(Cp)^i \leq \alpha = 1/(1+\epsilon/2)^i$, and
so the third piece of the recursion can be employed.  Reset 
$\epsilon$ to be $\epsilon/2$, $C$ to be $C(1 + \epsilon/2)$, and return
to the earlier stage until once again $i = 0$ or $i \geq 4.6\epsilon^{-1}$.
Continue until termination occurs.
Theorem 1 of~\cite{huber2016a} showed that the expected running time
of this algorithm was at most $9.5C\epsilon^{-1}$.  

\begin{lemma}
The algorithm is a correct Bernoulli factory.
\end{lemma}

\begin{proof}
To show global correctness, it suffices to first show local correctness
for the three pieces of the recursion.  Since Bernoulli distributions
are determined by their mean, that is equivalent to verifying
\begin{align*}
(Cp)^i &= p(Cp)^{i-1}+(1-p)(Cp)^{i-1}(C-1)p/(1-p) \\
(Cp)^i\frac{(C-1)p}{1-p} &=
  \frac{C-1}{C} (Cp)^{i+1} + \frac{1}{C} (Cp)^{i-1} \frac{(C-1)p}{1-p} \\
g(p) &= \alpha \cdot \alpha^{-1} g(p) + (1 - \alpha) \cdot 0
\end{align*}
Each of these results is straightforward to verify.

Since the expected running time of the algorithm is finite, the algorithm
terminates with probability 1, and so the FTPS immediately gives that 
the algorithm is correct. 
\end{proof}

\end{example}

\section{Proof of the FTPS}
\label{SEC:proof}

Let $X$ denote the output of $\textsc{Perfect\_Simulation}(\pi,0)$.  Let $T$
denote the largest value of $i$ attained during recursive calls to 
the algorithm.  Then the 
assumption that the algorithm halts with probability 1 is equivalent to 
saying that the 
probability $T$ is finite is 1.  The following tells us how close
the output distribution is to the target after a finite number of steps.

\begin{lemma}
Suppose equation~\eqref{EQN:FTPS} holds, and for all $i$ let 
$Y_i \sim \pi_{U_i}$.  Then 
for all $i$ and measurable $C$,
\begin{equation}
\label{EQN:finite}
\pi(C) = \prob(X \in C,T < i) + 
  \prob(f_i(U_i) \in C,T = i) + \prob(g_i(Y_i,U_i) \in C,T > i).
\end{equation}
\end{lemma}

\begin{proof}  The proof proceeds by induction.
Start with the $i = 0$ case.  Then always $T \geq 0$, so the 
first term on the right hand side is 0.  For $T = 0$, it must hold
that $U_0 \in A_0$.  
Since $\pi_0 \sim \pi$, 
equation~\eqref{EQN:finite} becomes
\[
\pi_0(C) = \prob(f(U_0) \in C,U_0 \in A_0) + 
  \prob(g_0(U_0,Y_0) \in C,U_0 \notin A_0) = \prob(X_0 \in C).
\]
By equation~\eqref{EQN:FTPS} this holds.

Our induction hypothesis assumes~\eqref{EQN:finite} holds 
for $i$, and consider what happens with $i + 1$:
\begin{align*}
\prob(X \in C,T < i+1) &= 
  \prob(X \in C,T < i) + \prob(X \in C, T = i) \\
 &= \prob(X \in C,T < i) + \prob(f_i(U_i) \in C,T = i) \\
 &= \pi(C) - \prob(g_{i}(Y_i,U_i) \in C,T > i)
\end{align*}
where the last step is our induction hypothesis.  Rearranging gives
\[
\pi(C) = \prob(X \in C,T < i+1) + 
 \prob(g_{i}(Y_i,U_i) \in C,T > i).
\]

To understand the second term on the right, note 
$\pi_{U_i} \sim \pi_{i+1}$, so by~\eqref{EQN:finite}
\[
Y_{i} \sim X_{i+1} = f_{i+1}(U_{i+1})\ind(U_{i+1} \in A_{i+1}) + 
            g_{i+1}(Y_{i+1},U_{i+1})\ind(U_{i+1} \notin A_{i+1}).
\]
That implies that 
\[
\prob(g_i(Y_i,U_i)\in C,T > i) =
 \prob(f_{i+1}(U_{i+1}) \in C,T = i) + 
 \prob(g_{i+1}(Y_{i+1},U_{i+1}) \in C,T > i),
\]
which completes the induction.
\end{proof}

This leads to a simple bound on the output probabilities.

\begin{lemma}
For all measurable $C$ and $i$,
\begin{equation}
\label{EQN:bound}
\prob(X \in C,T < i) \leq \pi(C) \leq \prob(X \in C,T < i) +
 \prob(T \geq i). 
\end{equation}
\end{lemma}

\begin{proof}
The two rightmost terms in~\eqref{EQN:finite} are bounded below by 1,
and above by $\prob(T \geq i)$, which gives the bound.
\end{proof}

With this bound in hand, the FTPS can now be proved.

\begin{proof}[Proof of the FTPS]
Let $C$ be any measurable set.
Simply take the limit as $i$ goes to infinity of~\eqref{EQN:bound}.  
If $\prob(T < \infty) = 1$, then by
the Dominated Convergence Theorem, this gives
\[
\prob(X \in C) \leq \pi(C) \leq \prob(X \in C),
\]
which implies $\prob(X \in C) = \pi(C).$
\end{proof}


\begin{thebibliography}{10}

\bibitem{asmussengt1992}
S.~Asmussen, P.~W. Glynn, and H.~Thorisson.
\newblock Stationarity detection in the initial transient problem.
\newblock {\em ACM Trans. Modeling and Computer Simulation}, 2(2):130--157,
  1992.

\bibitem{beskospr2006}
A.~Beskos, O.~Papspiliopoulous, and G.~O. Roberts.
\newblock Retrospective exact simulation of diffusion sample paths with
  applications.
\newblock {\em Bernoulli}, 12(6):1077--1098, 2006.

\bibitem{huber2006a}
M.~Huber.
\newblock Exact sampling from perfect matchings of dense regular bipartite
  graphs.
\newblock {\em Algorithmica}, 44:183--193, 2006.

\bibitem{huber2008b}
M.~Huber.
\newblock Perfect simulation with exponential tails.
\newblock {\em Random Structures Algorithms}, 33(1):29--43, 2008.

\bibitem{hubertoappearc}
M.~Huber.
\newblock A {B}ernoulli mean estimate with known relative error distribution.
\newblock {\em Random Structures Algorithms}, 2016.
\newblock {a}rXiv:1309.5413. To appear.

\bibitem{huber2016a}
M.~Huber.
\newblock Nearly optimal {B}ernoulli factories for linear functions.
\newblock {\em Combin. Probab. Comput.}, 25(4):577--591, 2016.
\newblock {a}rXiv:1308.1562.

\bibitem{huber2008a}
M.~Huber and J.~Law.
\newblock Fast approximation of the permanent for very dense problems.
\newblock In {\em Proc. of 19th ACM-SIAM Symp. on Discrete Alg.}, pages
  681--689, 2008.

\bibitem{huber2015b}
Mark~L. Huber.
\newblock {\em Perfect {S}imulation}.
\newblock Number 148 in Chapman \& Hall/CRC Monographs on Statistics \& Applied
  Probability. CRC Press, 2015.

\bibitem{jerrums1989}
M.~Jerrum and A.~Sinclair.
\newblock Approximating the permanent.
\newblock {\em J. Comput.}, 18:1149--1178, 1989.

\bibitem{keaneo1994}
M.~S. Keane and G.~L. O'Brien.
\newblock A {B}ernoulli factory.
\newblock {\em ACM Trans. Modeling and Computer Simulation}, 4:213--219, 1994.

\bibitem{latuszynskikpr2011}
K.~{\L}atuszy\'nski, I.~Kosmidis, O.~Papspiliopoulos, and G.O. Roberts.
\newblock Simulating events of unknown probabilities via reverse time
  martingales.
\newblock {\em Random Structures Algorithms}, 38(4):441--452, 2011.

\bibitem{nacup2005}
S.~Nacu and Y.~Peres.
\newblock Fast simulation of new coins from old.
\newblock {\em Ann. Appl. Probab.}, 15(1A):93--115, 2005.

\bibitem{proppw1996}
J.~G. Propp and D.~B. Wilson.
\newblock Exact sampling with coupled {M}arkov chains and applications to
  statistical mechanics.
\newblock {\em Random Structures Algorithms}, 9(1--2):223--252, 1996.

\bibitem{valiant1979}
L.~G. Valiant.
\newblock The complexity of computing the permanent.
\newblock {\em Theoret. Comput. Sci.}, 8:189--201, 1979.

\bibitem{valiant1979b}
L.~G. Valiant.
\newblock The complexity of enumeration and reliability problems.
\newblock {\em SIAM Journal on Computing}, 8:410--421, 1979.

\bibitem{vonneumann1951}
J.~von Neumann.
\newblock Various techniques used in connection with random digits.
\newblock In {\em Monte Carlo Method}, Applied Mathematics Series 12,
  Washington, D.C., 1951. National Bureau of Standards.

\end{thebibliography}
\end{document}